\documentclass{article}
\usepackage{amsmath, amsthm, amssymb}
\usepackage[margin=0.8in]{geometry}
\usepackage{parskip}
\usepackage{comment}
\usepackage{hyperref}
\usepackage{graphicx, cleveref} 
\usepackage[dvipsnames]{xcolor}
\usepackage{subcaption, tikz}
\usepackage{wasysym}

\newcommand{\dic}{\vec\chi}
\newcommand{\diomega}{\vec\omega}

\renewcommand{\le}{\leqslant}

\newcommand{\itv}[1]{\tilde{[}#1\tilde{]}}
\newcommand{\intI}{\mathcal{I}}
\newcommand{\arrowL}{\overrightarrow{L}}
\newcommand{\vomega}{\vec{\omega}}
\newcommand{\vchi}{\vec{\chi}}

\newtheorem{lemma}{Lemma}
\newtheorem{cor}[lemma]{Corollary}
\newtheorem{corollary}[lemma]{Corollary}
\newtheorem{question}[lemma]{Question}

\newtheorem{theorem}[lemma]{Theorem}
\newtheorem{conj}[lemma]{Conjecture}
\newtheorem{definition}[lemma]{Definition}

\newcommand{\stef}[1]{{\color{blue}Stéphan: #1}}

\newcommand{\sophie}[1]{\textcolor{Emerald}{Sophie: #1}}

\title{Decomposing tournaments into comparability graphs}

\usepackage{authblk}

\author[4]{Pierre Aboulker\thanks{This research was supported by the ANR project GODASse ANR-24-CE48-4377 and by the group Casino/ENS Chair on Algorithmics and Machine Learning.}}
\author[1]{Logan Crew}
\author[6]{Julien Duron}
\author[1]{Xinyue Fan}
\author[3]{Hugo Jacob}
\author[4]{Rémy Kimbrough}
\author[1]{Hidde Koerts}
\author[2]{Benjamin Moore}
\author[1]{Sophie Spirkl\thanks{We acknowledge the support of the Natural Sciences and Engineering Research Council of Canada (NSERC), [funding reference numbers RGPIN-2020-03912 and RGPIN-2022-03093].
Cette recherche a \'et\'e financ\'ee par le Conseil de recherches en sciences naturelles et en g\'enie du Canada (CRSNG), [num\'eros de r\'ef\'erence RGPIN-2020-03912 et RGPIN-2022-03093]. This project was funded in part by the Government of Ontario. Benjamin Moore also acknowledges the support of NSERC grant RGPIN-2025-07125.  This research was conducted while Spirkl was an Alfred P. Sloan Fellow. This research was undertaken, in part, thanks to funding from the Canada Research Chairs Program.}}
\author[5]{Stéphan Thomassé\thanks{Emails: pierreaboulker@gmail.com, (lcrew, xinyue.fan, hkoerts, sspirkl)@uwaterloo.ca, Ben.Moore@umanitoba.ca}}

\affil[1]{University of Waterloo, Department of Combinatorics and Optimization, Waterloo, Canada}

\affil[2]{University of Manitoba, Department of Mathematics, Winnipeg, Canada}

\affil[3]{Université de Montpellier, CNRS, LIRMM, Montpellier, France}

\affil[4]{DIENS, École normale supérieure, CNRS, PSL University, Paris, France}

\affil[5]{Univ. Lyon, ENS de Lyon, UCBL, CNRS, LIP, France.}

\affil[6]{Institute of Informatics, University of Warsaw, Poland}

\date{\today}

\begin{document}

\maketitle

\begin{abstract}
    In this note, we introduce the \emph{partial order decomposition number} of a digraph $D$, denoted $pod(D)$, defined as the minimum integer $k$ such that $A(D)=A(P_1)\cup\cdots\cup A(P_k)$, where $P_1,\ldots,P_k$ are partial orders on $V(D)$. We prove that $\dic(D)\le \diomega(D)^{pod(D)}$ for every digraph $D$. In particular, every class of digraphs with bounded $pod$ is polynomially $\dic$-bounded. We apply this to tournaments, showing that if $\mathcal C$ is a class of tournaments with bounded dichromatic number, then the closure of $\mathcal C$ under substitution is polynomially $\dic$-bounded, thereby making progress on a question of Aubian, Charbit, Lopes, and the first author. 
    As further applications of  $pod$, we prove that poset tournaments of bounded dimension are $\dic$-bounded, derive polynomial lower bounds on the directed clique number of an explicit family of tournaments, thereby answering a conjecture of Gutowski and Rams, and show that tournaments with bounded $pod$ have bounded domination number.
\end{abstract}
\tableofcontents

\section{Introduction}

This note focuses on a new digraph parameter, called partial order decomposition number.  
Recall that a \emph{partial order} $\leq$ on a set $S$ is a transitive, reflexive, and anti-symmetric relation. We will abuse terminology and say that a digraph $P = (V,A)$ which is acyclic, and furthermore satisfies that when $(u,v) \in A(P)$ and $(v,w) \in A(P)$ then $(u,w) \in A(P)$, is a \textit{partial order}. Thus, we are interpreting partial orders as transitive  acyclic digraphs. Such digraphs are also called \emph{comparability digraphs}.

Let $D$ be a digraph. A \emph{partial order decomposition} of $D$ is a collection $P_1,\ldots,P_k$ of partial orders on $V(D)$ such that
$A(D)=A(P_1)\cup\cdots\cup A(P_k)$.
The \emph{partial order decomposition number} of a digraph $D$, denoted by \emph{$pod(D)$}, is the minimum integer $k$ such that $D$ admits a partial order decomposition with $k$ parts.

The interest of the parameter $pod$ lies in its relation to the dichromatic number and the clique number of digraphs.

Before defining these two notions, let us recall their classic undirected counterpart. 
Given an undirected graph $G$, we denote by $\omega(G)$ the size of a maximum clique of $G$. A \textit{$k$-colouring} is a map $f:V(G) \to \{1,\ldots,k\}$ such that for all $xy \in E(G)$, we have $f(x) \neq f(y)$. The \textit{chromatic number} of $G$, denoted by $\chi(G)$, is the minimum $k$ such that $G$ is $k$-colourable (i.e. admits a $k$-colouring).

Digraphs in this paper have no parallel edges or loops. The existence of an arc from $a$ to $b$ in a digraph $D$ is denoted $a \to_D b$. We drop the subscript when the digraph is clear from context. 
For a digraph $D$, a \textit{$k$-dicolouring} is a map $f:V(D) \to C$ such that $|C| = k$ for all $i \in C$, the vertices coloured $i$ induce an acyclic digraph.  The \textit{dichromatic number} of $D$, denoted $\dic(D)$, is the minimum $k$ such that $D$ has a $k$-dicolouring. It was first introduced by Neumann-Lara in~\cite{neumann}.

Extending the notion of clique number to digraphs is slightly more cumbersome, as it requires introducing backedge graphs. Given a digraph $D$ and a total ordering $<$ on $V(D)$, the \textit{backedge graph of $D$ with respect to $<$}, denoted $B(D,<)$ (or simply $B$ when $D$ and $<$ are clear from context), is the undirected graph with $V(B)=V(D)$ and edge set $E(B)=\{vw : v<w,\ w\to_D v\}$. We say that a graph $G$ is a backedge graph of a digraph $D$ if there is an ordering $<_G$ of $V(D)$ such that $B(D,<_G)$ is isomorphic to $G$. Note that $B(D,<)$ comes naturally equipped with the ordering $<$, and we will often use it as the ordered graph $(B(D,<),<)$.

The following observation shows that the dichromatic number of a digraph  is completely determined by the backedge graphs of the digraph. 
It is straightforward that every independent set  of $B(D,<)$ induces an acyclic subdigraph of $D$. As a consequence, we have that $\dic(D) \le \chi(B(D, <))$. 
Conversely, by taking an ordering built from a $\dic(D)$-dicolouring,
that is taking colour classes one after the other, and ordering each colour class in a topological ordering, we get that:
\[\vec{\chi}(D) = \min\{\chi(B(D,<)):\, < \text{ is a total ordering of } V(D)\}.
\] 

This naturally leads to the following notion of clique number for a digraph, introduced in \cite{original}. 
The \textit{clique number of a digraph $D$}, denoted $\vec{\omega}(D)$, is the minimum clique number of any backedge graph of $D$. That is, 
\[\vec{\omega}(D) = \min\{\omega(B(D,<)):\, < \text{ is a total ordering of } V(D)\}\]

A total ordering $<$ of $V(D)$ is an \textit{$\diomega$-ordering} of $D$ if $\omega(B(D, <)) = \diomega(D)$.

More generally, see Section~3 of \cite{Digraph_parameter} for a discussion of digraph parameters defined from graph parameters via backedge graphs. 

The following theorem is the main reason why the parameter $pod$ is useful in our setting.
\begin{theorem}\label{thm:pod-main}
For every digraph $D$, we have $\dic(D)\le \diomega(D)^{pod(D)}$.
\end{theorem}

A natural application of the above inequality is to $\dic$-bounded classes of digraphs. Let us first recall the corresponding undirected notion.
It is well known that there exist graphs $G$ with $\omega(G)=2$ and $\chi(G)$ arbitrarily large. Thus, it is natural to study graph classes for which chromatic number is controlled by clique number.
More specifically, the following question has been widely studied: for which graph classes $\mathcal{C}$ does there exist a function $f$ such that for all $G\in\mathcal{C}$, we have $\chi(G)\le f(\omega(G))$? Such graph classes are called \textit{$\chi$-bounded} (see~\cite{scott2020survey} for a survey on this topic). In the case where $f$ is a polynomial, they are called \textit{polynomially $\chi$-bounded}, and there exist $\chi$-bounded classes which are not polynomially $\chi$-bounded \cite{separating}.

This notion extends naturally to digraphs, with $\dic$ and $\diomega$ playing the roles of $\chi$ and $\omega$. Let $\mathcal{C}$ be a class of digraphs. We say $\mathcal{C}$ is \textit{$\dic$-bounded} if there exists a function $f$ such that for every digraph $D \in \mathcal{C}$, we have $\dic(D) \leq f(\vec{\omega}(D))$. 
As before, if $f$ can be taken to be a polynomial, then we say that $\mathcal{C}$ is \textit{polynomially $\dic$-bounded.}

As an immediate corollary of Theorem~\ref{thm:pod-main}, we get that digraph classes with bounded $pod$ are polynomially $\dic$-bounded.

\begin{corollary}\label{cor:poly-chi-pod}
Let $m$ be a positive integer, and let $\mathcal{C}$ be a class of digraphs such that
$pod(D)\le m$ for every $D\in\mathcal{C}$. Then $\mathcal{C}$ is polynomially $\dic$-bounded.
More precisely, for every $D\in\mathcal{C}$, we have $\dic(D)\le \diomega(D)^m$. 
\end{corollary}

We use this to make progress on a conjecture of Aubian, Charbit, Lopes and the first author~\cite{original}. The notion of substitution is standard, and will be recalled in Section~\ref{sec:tournamentsofboundeddic}. 
Given a class of digraphs $\mathcal{C}$ we define $\mathcal{C}^{subst}$ to be the \textit{closure of $\mathcal{C}$ under substitution}. 
A \emph{tournament} is an orientation of a complete graph. 

\begin{conj}[\cite{original}]
\label{conj:substitution}
If a class of tournaments $\mathcal{C}$
is polynomially $\dic$-bounded, then so is $\mathcal{C}^{subst}$.
\end{conj}

This conjecture is motivated from the graph version, which says that the graph analogy for substitution preserves polynomial $\chi$-boundness \cite{substitution}.
We prove Conjecture \ref{conj:substitution} for the class of tournaments which have bounded dichromatic number:
\begin{theorem}
\label{thm:boundeddichromatic}
Let $\mathcal{C}$ be a class of tournaments such that each tournament in $\mathcal{C}$ has dichromatic number at most $k$. Then for each $T \in \mathcal{C}^{subst}$, we have  $\dic(T) \leq (\diomega(T))^{2\lceil \log_2(k) \rceil +1}$. In particular, the class $\mathcal{C}^{subst}$ is polynomially $\dic$-bounded. 
\end{theorem}


It is natural to ask  whether bounded $pod$ characterizes polynomially $\dic$-bounded classes of  tournaments. 
We  exhibit a class of tournaments with unbounded $pod$, namely crossing tournaments. In a follow-up paper, some of the authors of this note show that this class is nevertheless polynomially $\dic$-bounded, thereby showing that bounded $pod$ does not characterize polynomially $\dic$-bounded classes of tournaments.


We also prove that a subclass of poset tournaments (defined in Section~\ref{sec:poset_tournaments}), namely the class of poset tournaments with bounded dimension is $\dic$-bounded. 

We then show how to derive lower bounds on the clique number of tournaments from the inequality $\dic \leq \diomega^{pod}$. We illustrate this method by showing that a particular family of tournaments, denoted $(B_k)_{k \geq 0}$ and introduced in \Cref{sec:lower_bound}, has clique number at least polynomial in $|V(B_k)|$. This leads to a positive answer to the following conjecture. 

\begin{conj}[Gutowski and Rams, \cite{gutowski2026notecomplexitydirectedclique}]\label{conj:Gutowski}
   There exists an explicit family of tournaments $(T_n)_{n\ge 1}$ such that $|V(T_n)|=\Theta(n)$ and $\diomega(T_n)=\Omega(n^{1/100})$  
\end{conj}


The structure of the paper is as follows. We prove the inequality $\dic \leq \diomega^{pod}$ in Section~\ref{sec:main_eq}. Corollary~\ref{cor:poly-chi-pod} and Theorem~\ref{thm:boundeddichromatic} are proved in Section~\ref{sec:tournamentsofboundeddic}, and the result on crossing tournaments and poset tournaments are  proved respectively in Section~\ref{sec:crossingtournaments} and Section~\ref{sec:poset_tournaments}, and the lower bounds on the clique number of tournaments are discussed in Section~\ref{sec:lower_bound}. Finally, we end the paper with Section~\ref{sec:open_questions}, which contains two additional short results on $pod$ and a few open questions on tournaments and $pod$. The two results show that acyclic digraphs can have arbitrarily large $pod$, and that tournaments with bounded $pod$ have bounded domination number.

\section{Proof that $\dic \leq \diomega^{\,pod}$}\label{sec:main_eq}

A graph $G$ is a \textit{comparability graph} if it admits a transitive orientation.  
The next lemma follows from the well-known fact that the intersection of two partial orders is a partial order. We include a proof for completeness: 

\begin{lemma}\label{lem:compgraphbackedge}
Let $D$ be a partial order digraph. Then every backedge graph of $D$ is a comparability graph.
\end{lemma}

\begin{proof}
Let $<$ be a total order on $V(D)$, and let $B$ be the corresponding backedge graph.
We orient each edge of $B$ as in $D$, that is, if $x<y$ and $xy\in E(B)$, we orient the edge $xy$ of $B$ as $y \to x$. 
We claim that this is a transitive orientation of $B$.
Indeed, suppose that $a \to b$ and $b \to c$ are oriented edges of $B$.
Since the orientation of $B$ is inherited from $D$, we have $a \to b$ and $b \to c$ in $D$.
Moreover, because these edges belong to the backedge graph, both arcs are backward with respect to $<$, so $c<b<a$.
Since $D$ is transitive, it follows that $a \to c$ in $D$.
As $c<a$, 
we have that $ac\in E(B)$, and is oriented as $a \to c$. 
Thus the chosen orientation of $B$ is transitive and $B$ is a comparability graph.
\end{proof}


This is useful because comparability graphs are perfect.

\begin{theorem}[\cite{berge}]
\label{thm:compareperfect}
    Comparability graphs are perfect: if $G$ is a comparability graph, then every induced subgraph $H$ of $G$ satisfies $\chi(H)=\omega(H)$.
\end{theorem}

We are now ready to prove the main result on $pod$, that is the equation linking it with $\dic$ and $\diomega$. 

\newtheorem*{thm:pod-main}{Theorem \ref{thm:pod-main}}
\begin{thm:pod-main}
For every digraph $D$, we have $\dic(D)\le \diomega(D)^{pod(D)}$.
\end{thm:pod-main}

\begin{proof}
    Let $D$ be a digraph. Set $m=pod(D)$ and let $P_{1},\ldots,P_{m}$ be a partial order decomposition of $D$.  
    Let $<$ be an $\diomega$-ordering of $D$ and $B=B(D,<)$. So $\omega(B) = \diomega(D)$. For $i \in \{1, \ldots, m\}$, let $B_i=B(P_i, <)$. By Lemma~\ref{lem:compgraphbackedge}, each $B_i$ is a comparability graph, and is thus perfect.  Moreover, since
    $A(D)=A(P_1)\cup\cdots\cup A(P_m)$, we have
    $E(B)=E(B_1)\cup\cdots\cup E(B_m)$.

    Now, there exist proper colourings $f_i:V(B)\to \{1,\ldots,\omega(B_i)\}$ of each $B_i$.
    Since $\omega(B_i)\le \omega(B)$ for every $i$, the colouring
    $f:V(B)\to \{1,\ldots,\omega(B)\}^m$ defined by $f(v)=(f_1(v),\dots,f_m(v))$
    is a proper colouring of $B$ using at most $\omega(B)^m=\diomega(D)^m$ colours.
    Now, since the chromatic number of a backedge graph is always at most the dichromatic number of the digraph, we have $\dic(D)\le \chi(B)\le \diomega(D)^m$, as desired.
\end{proof}


\section{Polynomially $\dic$-bounded classes and substitution}
\label{sec:tournamentsofboundeddic}


The goal of this section is to prove Theorem \ref{thm:boundeddichromatic}. 
Let us first give necessary definitions. 

Given a digraph $D$, and $X,Y \subseteq V(D)$,  we say $X \Rightarrow Y$ if, for every $x \in X$ and $y \in Y$, we have $xy \in A(D)$. If $X$ or $Y$ is a singleton, we may drop the brackets to simplify the notation.

Given two digraphs $D_1$ and $D_2$ with disjoint vertex sets, and a vertex
$u \in V(D_1)$, we say that a digraph $D$ is obtained by \textit{substituting $D_2$ for $u$ in $D_1$} if the following hold:

\begin{itemize}
    \item $V(D) = (V(D_1) \setminus \{u\}) \cup V(D_2)$,
    \item $D[V(D_1) \setminus \{u\}] = D_1 \setminus u$,
    \item $D[V(D_2)] = D_2$,
    \item for all $v \in V(D_1) \setminus \{u\}$ if $v \to_{D_1} u$ (resp. $u \to_{D_1} v$, resp. $u$ and $v$ are non-adjacent in $D_1$), then $v \Rightarrow V(D_2)$ (resp. $V(D_2) \Rightarrow v$, resp. there are no arcs between $v$ and $V(D_2)$) in $D$.
\end{itemize}

We first prove that the class of tournaments with $pod$ at most $m$ is closed under substitution.

\begin{lemma}\label{lem:compgraphsubs}
Let $\mathcal{C}$ be a class of tournaments such that every $T\in\mathcal{C}$ satisfies $pod(T)\le m$.
Then every tournament $T\in\mathcal{C}^{subst}$ also satisfies $pod(T)\le m$.
\end{lemma}

\begin{proof}
    It suffices to show that if $T$ and $T'$ are both tournaments such that $pod(T) \leq m$ and $pod(T') \leq m$, then the same holds for all substitutions of $T'$ for a vertex of $T$. 

    Consider a substitution of $T'$ for a vertex $v$ of $T$ and let the resulting tournament be $T_{v \rightsquigarrow T'}$.  Let $P_1, \dots, P_m$ and $P'_1, \dots, P'_m$ be partial order decompositions of respectively $T$ and $T'$. We set $Q_i$ to be the spanning subdigraph of $T_{v \rightsquigarrow T'}$ obtained from substituting $P'_i$ for vertex $v$ in $P_i$. More formally: 
    \begin{align*}
    A(Q_i) ={}& A(P'_i) \cup \{e \in A(P_i) \mid v \notin e\}
 \cup \{(x,y) : x \in V(T'),\ (v,y) \in A(P_i)\}  \cup \{(y,x) : x \in V(T'),\ (y,v) \in A(P_i)\}.
\end{align*}

    Clearly $A(T_{v \rightsquigarrow T'}) = A(Q_1) \cup \dots \cup A(Q_m)$. It remains to show that each $Q_i$ is transitive. Let $a,b,c \in V(T_{v \rightsquigarrow T'})$ be three vertices such that $a \rightarrow_{T_{v \rightsquigarrow T'}} b$ and $b \rightarrow_{T_{v \rightsquigarrow T'}} c$. If $a, b, c \in V(T')$, then $(a,b), (b,c) \in A(P'_i)$, and $P'_i$ is transitive, so $(a,c) \in A(P'_i) \subseteq A(Q_i)$. If $|\{a, b, c\} \cap V(T')| \leq 1$, then the arcs corresponding to $(a, b)$ and $(b, c)$ in $P_i$ guarantee the existence of $(a,c)$ because $P_i$ is transitive. 

    Finally, observe that if $|\{a,b,c\} \cap V(T')| = 2$ then $b \in V(T')$ by definition of the substitution. Indeed, we have $x \Rightarrow V(T')$ for every $x \in V(T_{v \rightsquigarrow T'}) \setminus V(T')$. Then, exactly one of $a$ and $c$ is in $V(T')$ and in either case we have $a \to_{T_{v \rightsquigarrow T'}} c$ by definition of the substitution.
\end{proof}

A particular special case of tournaments with bounded $pod$ are tournaments with bounded $\dic$. In particular, every digraph with a bipartite underlying undirected graph has $pod$ at most 2, which can be used iteratively to show the following. 

\begin{lemma}\label{lem:boundedchrom}

Let $T$ be a tournament with dichromatic number $k$ and colour classes $V_1, \dots, V_k$ arising from a $k$-dicolouring of $T$. For each $i$, let $A(V_i)$ denote the set of arcs of $T$ with both endpoints in $V_i$. Then   $T \backslash \{A(V_1), \dots, A(V_k)\}$ decomposes (disjointly) into at most $2\lceil \log_2(k)\rceil$ comparability digraphs.
In particular, \\ $pod(T \backslash \{A(V_1), \dots, A(V_k)\}) \leq 2\lceil \log_2(k)\rceil$.
    
\end{lemma}

\begin{proof}
    Throughout the proof let $V = V(T)$ for clarity. Let $f(k)$ be the smallest number such that every tournament with dichromatic number $k$ and colour classes $V_1, \dots, V_k$ allows $T \backslash (A(V_1) \cup \dots \cup A(V_k))$ to be decomposed into at most $f(k)$ comparability digraphs. We aim to show that $f$ exists, and that $f(k) \leq 2\lceil \log_2(k)\rceil$ for $k \geq 1$.

   We proceed by strong induction on $k$. It suffices to consider the cases when $k$ is a power of $2$. The case $k = 1$ is straightforward as $V_1 = V$, and so $T \setminus A(V_1)$ has no edges. 
   
   Now let $k = 2^t$ for some $t \geq 1$. Let $ m= 2^{t-1}$. Let $T_{\alpha} =  T[V_1 \cup \dots \cup V_m]$ and $T_{\beta} = T[V_{m+1} \cup \dots \cup V_k]$ be the subtournaments of $T$ induced by $V_1 \cup \dots \cup V_m$ and $V_{m+1} \cup \dots \cup V_k$. respectively. We apply strong induction to $T_{\alpha}$ and $T_{\beta}$ to decompose $T_\alpha \setminus (A(V_1) \cup \dots \cup A(V_m))$ and $T_\beta \setminus (A(V_{m+1}) \cup \dots \cup A(V_k))$ into $2 \lceil \log_2 m \rceil = 2(t-1)$ comparability digraphs each. 
   Since the disjoint union of two comparability digraphs is a comparability digraph, we may partition the arcs of $(A(T_\alpha) \cup A(T_\beta)) \setminus (A(V_1) \cup \dots \cup A(V_k))$
   into $2(t-1)$ comparability digraphs $D_\alpha^1 \cup D_\beta^1,\dots,D_\alpha^{2(t-1)} \cup D_\beta^{2(t-1)}$, where $D_\alpha^1,\dots,D_\alpha^{2(t-1)}$ and $D_\beta^1,\dots, D_\beta^{2(t-1)}$ are obtained by induction hypothesis on $T_\alpha$ and $T_\beta$, respectively.

   Now we create two more comparability digraphs: One containing all arcs from $V(T_\alpha)$ to $V(T_\beta)$, and one containing all arcs from $V(T_\beta)$ to $V(T_\alpha)$. Note that every vertex is a source or a sink with respect to each of these sets, so these are indeed comparability digraphs. This yields the desired decomposition using $2t$ comparability digraphs. 
   \end{proof}

   \begin{cor}\label{cor:boundedchrom}

   Let $T$ be a tournament with dichromatic number $k$. Then $T$ may be decomposed (disjointly) into at most $2\lceil \log_2(k)\rceil+1$ comparability digraphs. In particular, $pod(T) \leq 2\lceil \log_2(k)\rceil+1$.
       
   \end{cor}

   \begin{proof}
By Lemma \ref{lem:boundedchrom}, if $V_1, \dots, V_k$ are colour classes of a $k$-colouring of $T$, then $A(T) \backslash (A(V_1) \cup \dots \cup  A(V_k))$ may be decomposed into at most $2\lceil \log_2(k)\rceil$ comparability digraphs. Then the result follows by noting that $(V(T), A(V_1) \cup \dots \cup A(V_k))$, as a disjoint union of acyclic tournaments, is a comparability digraph as well. \end{proof}

With these lemmas in hand, it is straightforward to prove the  main result of this section.
\newtheorem*{thm:boundeddichromatic}{Theorem \ref{thm:boundeddichromatic}}
\begin{thm:boundeddichromatic}
Let $\mathcal{C}$ be a class of tournaments such that each tournament in $\mathcal{C}$ has dichromatic number at most $k$. Then for each $T \in \mathcal{C}^{subst}$, we have  $\dic(T) \leq (\diomega(T))^{2\lceil \log_2(k) \rceil +1}$. In particular, the class $\mathcal{C}^{subst}$ is polynomially $\dic$-bounded. 
\end{thm:boundeddichromatic}

   \begin{proof}
       By Corollary \ref{cor:boundedchrom}, each tournament $T \in \mathcal{C}$ satisfies $pod(T) \leq 2\lceil \log_2(k) \rceil +1$. By Lemma \ref{lem:compgraphsubs}, each tournament $T \in \mathcal{C}^{subst}$ also satisfies $pod(T) \leq 2\lceil \log_2(k) \rceil +1$. Then, by Corollary \ref{cor:poly-chi-pod}, the result follows.
   \end{proof}

   By modifying the above lemmas and theorems to the case where $\mathcal{C}$ is a polynomially $\dic$-bounded class of tournaments, say with $\vchi(T) \leq (\vomega(T))^r$ for each $T \in \mathcal{C}$, it is possible to show that then $\mathcal{C}^{subst}$ is $\dic$-bounded by  $$\vomega(T)^{2\lceil r\log_2 \vomega(T) \rceil + 1},$$ improving the single-exponential bound of \cite{original} to a ``quasi-polynomial'' one. However, it appears that a different technique may be required to prove (or disprove) that substitution fully preserves polynomial $\dic$-boundedness in tournaments.


\section{Crossing tournaments do not have bounded pod}
\label{sec:crossingtournaments}

In the previous section, we showed that digraphs with bounded $pod$ are polynomially $\dic$-bounded, and that whenever a class $\mathcal{C}$ of tournaments has bounded $pod$, then $\mathcal{C}^{subst}$ is polynomially $\dic$-bounded. It is therefore natural to ask whether the converse of these two statements holds. In this section, we show that it does not: there exist acyclic digraphs with arbitrarily large $pod$, and there is a class of tournaments, namely crossing tournaments, with unbounded $pod$. In a future paper, some of the authors show that crossing tournaments are nevertheless polynomially $\dic$-bounded, completing the proof that the converse of Theorem~\ref{thm:boundeddichromatic} is false.

We now turn our attention to crossing tournaments. Let us first recall some definitions about interval graphs.  

An \textit{interval system} $\mathcal{I}$ is a set of open intervals in $\mathbb{R}$ such that no two elements of $\mathcal{I}$ share an endpoint. For an interval $I \in \mathcal{I}$, we denote its left endpoint by $L(I)$ and its right endpoint by $R(I)$. The interval system $\mathcal{I}$ is naturally equipped with two (not necessarily distinct) total orders: $<_L$, the ordering by left endpoint, and $<_R$, the ordering by right endpoint. There is also a natural partial order $\leq_{\mathcal I}$ defined by $I <_{\mathcal I} J$ if and only if $R(I) < L(J)$; that is, if and only if $I$ and $J$ are disjoint and $I$ lies to the left of $J$ in $\mathbb{R}$.

The \textit{interval graph} $G_{\mathcal I}$ of an interval system $\mathcal I$ is the ordered graph with vertex set $\mathcal I$, ordered by $<_L$, in which two intervals $I,J \in \mathcal I$ with $I <_L J$ are adjacent if and only if $R(I) > L(J)$.

A tournament $T$ is a \textit{crossing tournament} if there exists an interval system $\mathcal I$ and a total ordering $<$ of $V(T)$ such that the ordered graph $(B(T,<),<)$ is isomorphic to the ordered interval graph $(G_{\mathcal I},<_L)$. Crossing tournaments were introduced in \cite{crossing} as tournaments with unbounded $\diomega$.

Let us now prove  that ``universal" crossing tournaments have unbounded $pod$.

We define the following notation.
Intuitively, this corresponds to using closed intervals with additional tie-breaking rules to order their endpoints.
\begin{definition}\label{def:intGeneral}
    For positive integers $1 \leq i < j \leq n$, let \\ $$\itv{i,j} = \left(i-\frac{1}{j+2}, j+\frac{1}{i+2}\right).$$ 
    Let $\intI_n$ be the interval system consisting of intervals $\{\itv{i,j}:1 \leq i < j \leq n\}$. Let $B_n$ be the ordered interval graph of $\intI_n$, and let $T_n$ be the tournament with backedge graph $B_n$.  
\end{definition}

Thus, the vertices of $B_n$ are the intervals $\itv{i,j} \in \intI_n$, whose left endpoints are increasing in lexicographic order on the pair $(i,j)$.

It is easy to check that two intervals $I_1 = \itv{i_1,j_1}$ and $I_2 = \itv{i_2,j_2}$ with $L(I_1) < L(I_2)$ are adjacent in $B_n$ if and only if $i_1 \leq i_2 \leq j_1$. Thus, these are exactly the cases in which pairs of intervals in $T_n$ have an arc $I_2 \rightarrow_{T_n} I_1$, and in all other cases we have an arc $I_1 \rightarrow_{T_n} I_2$.

The proof of the theorem below is adapted from a similar one of Gy{\'a}rf{\'a}s, Marits, and T{\'o}th \cite{perfectIntoComparability}, proving that in the undirected setting, interval graphs cannot be written as a union of a bounded number of comparability graphs. 

We need the following definitions. Given a digraph $D$, its \emph{underlying (undirected) graph} $UG(D)$ is defined as the graph with $V(UG(D)) = V(D)$ and $vw \in E(UG(D))$ if and only if $(v,w) \in A(D)$ or $(w,v) \in A(D)$. Given a digraph $D = (V,A)$, the \emph{line digraph} $\arrowL(D)$ has vertex set $A$ and arc set $\{(u,v)(v,w): u,v,w \in V, (u,v), (v,w) \in A\}$. Thus, whenever two arcs of $D$ meet head-to-tail, there is an arc between them in the natural direction in $\arrowL(D)$. We will use the following auxiliary result on line digraphs.

\begin{theorem}[Erd\H{o}s and Hajnal~\cite{EH66}, Harner and Entringer \cite{loglinedigraph}]\label{thm:logcoloring}
Given a digraph $D$, we have 
$$
\chi(UG(\arrowL(D))) \geq \log_2(\chi(UG(D)))
$$
where for $D$ a digraph, $UG(D)$ is the underlying undirected graph. 
\end{theorem}

\begin{theorem} \label{thm:notkcomp}
For every positive integer $k'$, there is a crossing tournament $T$ with $pod(T) \geq k'$.
\end{theorem}

\begin{proof}
    Let $T_n$ be as in Definition \ref{def:intGeneral} with $n \geq 2^{2^{k'}}$.

    Let $P_1, \dots, P_m$ be a partial order decomposition of $T_n$. Let $\phi: A(T_n) \rightarrow [m]$ map each arc $r$ of $T_n$ to the least $i$ such that $r \in A(P_i)$. Thus the sets $\phi^{-1}(i)$ for $i \in [m]$ partition $A(T_n)$.

    For each quadruple of positive integers $1 \leq a < b < c < d \leq n$, note that we have $\itv{a,b} \leftarrow \itv{b,c}$, $\itv{b,c} \leftarrow \itv{c,d}$, and $\itv{a,b} \rightarrow \itv{c,d}$ in $A(T_n)$, so these intervals form a cyclic triangle. It follows that for each such integer quadruple, we have $\phi(\itv{b,c}\itv{a,b}) \neq \phi(\itv{c,d}\itv{b,c})$. Thus, letting $G_n$ be the graph with $V(G_n) = \{(a,b,c): 1 \leq a < b < c \leq n\}$ and edge set $\{(a,b,c)(b,c,d): 1 \leq a < b < c < d \leq n \}$, we have that the graph $G_n$ is $m$-colourable, by giving vertex $(a,b,c)$ the colour $\phi(\itv{b,c}\itv{a,b})$, which is a proper colouring by the previous observation.

    Now, consider the transitive tournament $TT_n$ with vertices $[n]$ such that for each $i < j$ we have $i \rightarrow_{TT_n} j$. Then $\arrowL(TT_n)$ has vertices consisting of ordered pairs $(i,j)$ with $i < j$ with arcs $\{(i,j)(j,k): 1 \leq i < j < k \leq n\}$. Moreover, $\arrowL(\arrowL(TT_n))$ is isomorphic to a digraph with vertex set $\{(i,j,k): 1 \leq i < j < k \leq n\}$ and arc set $\{(i,j,k)(j,k,l): 1 \leq i < j < k < l \leq n\}$. In particular $UG(\arrowL(\arrowL(TT_n))$ is isomorphic to $G_n$ defined above. But then, by Theorem \ref{thm:logcoloring}, we have
    $$m \geq \chi(G_n) = \chi(UG(\arrowL(\arrowL(TT_n)))) \geq \log_2(\chi(UG(\arrowL(TT_n)) \geq \log_2(\log_2(\chi(UG(TT_n)))) = \log_2(\log_2(n)) \geq k'$$
    Since our initial partial order decomposition was arbitrary, we have $pod(T_n) \geq k'$.\end{proof}

\section{Poset tournaments}\label{sec:poset_tournaments}


We now turn our attention to poset tournaments.
A tournament $T$ is a \textit{poset tournament} if there exist a total ordering $<$ of $V(T)$ such that, for all $u < v < w$, if $u \to_T v$ and $v \to_T w$, then $u \to_T w$. In other words, the forward edges of $T$ for $<$ form a poset. In this case, it follows that $B(T,>)$ is a comparability graph.

Poset tournaments were first introduced in~\cite{dom_in_tournaments}, where they are conjectured to be precisely the tournaments $H$ such that $H$-free tournaments have bounded domination number.

We propose the following conjecture:

\begin{conj}\label{conj:poset_chi_bounded}
    The class of poset tournaments is $\dic$-bounded.
\end{conj}

We now recall the standard notion of (Dushnik–Miller) dimension of a poset. A \emph{linear extension} of a poset $P$ is a total ordering $<$ of $V(P)$ such that whenever $u <_P v$, we also have $u < v$. Given two posets $P_1,P_2$ on the same vertex set, their intersection, denoted $P_1 \cap P_2$, is the poset $P$ where $u <_P v$ if $u <_{P_1} v$ and $u <_{P_2} v$. (This also matches taking the intersection of the arc sets of corresponding oriented graphs).
The \emph{dimension} of a poset $P$, denoted $\dim(P)$, is the minimum integer $d$ such that there exist $d$ linear extensions $<_1,\dots,<_d$ of $P$ whose intersection is $P$, that is, such that for all distinct $u,v \in V(P)$, we have $u <_P v$ if and only if $u <_i v$ for every $i \in \{1,\dots,d\}$.

A \emph{witness} for a poset tournament $T$ is a pair $(P,<)$, where $P$ is a poset on $V(T)$ and $<$ is a linear extension of $P$, such that for all $u,v \in V(T)$ with $u<v$, we have $v\to_T u$ if and only if $u <_P v$.

The \emph{dimension} of a poset tournament $T$ is the minimum value of $\dim(P)$ over all witnesses $(P,<)$ for $T$.

As a partial result on  conjecture~\ref{conj:poset_chi_bounded}, we prove that the class of poset tournaments with bounded dimension is $\dic$-bounded. 

\begin{lemma}\label{lem:dim2pod} 
Let $T$ be a poset tournament. Then $pod(T) \le \dim(T) + 1$.
\end{lemma}

\begin{proof}
    Let $(P_0,<_0)$ be a witness for $T$ such that $\dim(P_0)=\dim(T)=d$. By definition of the dimension, there are linear extensions $<_1, ..., <_d$ of $P_0$ such that $P_0 = {} <_1 \cap \cdots \cap <_d$. For each $1 \le i \le d$, we define the poset $P_i = {} >_0 \cap <_i$. We show that $E(T) = E(P_0) \cup E(P_1) \cup \cdots \cup E(P_d)$.

    First, we show that $E(T) \subseteq E(P_0) \cup E(P_1) \cup \cdots \cup E(P_d)$. Let $u <_0 v$ be two vertices of $T$.
    If $u \to_T v$, then $u <_{P_0} v$ and $u \to_{P_0} v$.
    Otherwise, $v \to_T u$, which means that $u$ and $v$ are incomparable in $P_0$. This implies that there are $1 \le i, j \le d$ such that $u <_i v$ and $v <_j u$, hence $v \to_{P_j} u$.
    
    Conversely, we show that $E(P_0) \cup E(P_1) \cup \cdots \cup E(P_d) \subseteq E(T)$.
    If $u \to_{P_0} v$, then $u <_{P_0} v$, and $u <_0 v$ since $<_0$ is a linear extension of $P_0$, which means that $u \to_T v$.
    Now, if $u \to_{P_i} v$ for some $1 \le i \le d$, then $u >_0 v$ and $u <_i v$. This implies that $u$ and $v$ are not comparable in $P_0$, which in turn means that $u \to_T v$.
\end{proof}

\begin{corollary}
    For every integer $k$, the class of poset tournaments with dimension at most $k$ is $\dic$-bounded.
\end{corollary}


\section{Using partial order decompositions to lower bound the clique number of tournaments}\label{sec:lower_bound}

The inequality $\dic \leq \diomega^{pod}$ can  be rewritten as $\diomega \geq \dic^{1/pod}$, and thus yields lower bounds on the clique number of tournaments when $\dic$ and $pod$ are known. Let us illustrate this with  two examples. 

Let $B_0$ be the one-vertex tournament, and for $k \geq 1$, let $B_k$ be the tournament obtained from three disjoint copies $A$, $B$, and $C$ of $B_{k-1}$ by adding all possible arcs from $A$ to $B$, from $B$ to $C$, and from $C$ to $A$, i.e. $A \Rightarrow B, B \Rightarrow C, C \Rightarrow A$. 
Computing the clique number, or even a decent approximation of the clique number of $B_k$ appears to be quite difficult, whereas obtaining lower bounds on the dichromatic number and upper bounds on the partial order decomposition number is much easier, yielding a useful lower bound on its clique number. We now explain how to bound both the dichromatic number and the partial order decomposition number of $B_k$. 

Let $TT_k$ be the transitive tournament on $k$ vertices, and let $\vec C_3$ be the directed triangle. Set $\mathcal B=\{TT_1,TT_2,\vec C_3\}^{subst}$. Observe that $B_k\in\mathcal B$ for every $k$. Since $TT_1$, $TT_2$, and $\vec C_3$ all have $pod$ at most $3$, Lemma~\ref{lem:compgraphsubs} implies that $pod(B_k)\le 3$ for every $k$. Hence $\diomega(B_k)\ge \dic(B_k)^{1/3}$.

Now, let us prove that $\dic(B_k)\ge \frac32\,\dic(B_{k-1})$. In any dicolouring of $B_k$, each colour can appear in at most two of the three copies of $B_{k-1}$, since otherwise choosing one vertex of that colour in each copy yields a monochromatic directed triangle. Since each copy of $B_{k-1}$ requires at least $\dic(B_{k-1})$ colours, it follows that $\dic(B_k)\ge \frac32\,\dic(B_{k-1})$. By induction, we obtain $\dic(B_k)\ge (3/2)^k$. Altogether we get the following. 

\begin{theorem}\label{thm:diomega_Bk}
For every integer $k$, $\diomega(B_k)\ge (3/2)^{k/3}$.  
\end{theorem} 

Since $|V(B_k)|=3^k$, Theorem~\ref{thm:diomega_Bk} gives $\diomega(B_k) \geq |V(B_k)|^{1/8.13}$ and thus a positive answer to Conjecture~\ref{conj:Gutowski} as announced in the introduction. 

Next, we introduce a construction due to Nassar and Yuster \cite{nassar2019acyclic}, and show how to bound its clique number from below. For $t \in \mathbb{N}$, the tournament $G_t$ has vertex set $[t] \times [t]$, and we have $(i, j) \rightarrow (k, l)$ if and only if: 
\begin{itemize}
    \item $i \leq k$ and $j < l$; or
    \item $i < k$ and $j > l$; 
    \item $i > k$ and $j = l$. 
\end{itemize}
From this definition, it is immediate that $pod(G_t) \leq 3$, since each of the bullet points above is a partial order. In Lemma 3.1 (and the proof of Theorem 1.6) of \cite{nassar2019acyclic}, it is shown that no backedge graph of $G_t$ has an independent set of size $2t$. (In fact, the proof of Theorem 1.6 explicitly shows that every backedge graph of $G_t$ has a clique of size at least $|V(G_t)|^{1/8}/3$ and can be adapted to give a similar bound to ours. However, using $pod$ yields a much shorter proof.) Consequently, we have $\dic(G_t) \geq t^2 / (2t) = t/2$.  Since $pod(G_t) \leq 3$, we conclude the following. 
\begin{theorem} \label{thm:gt}
    For every $t \geq 1$, we have $\diomega(G_t) \geq \dic(G_t)^{1/3} = (t/2)^{1/3} \geq |V(G_t)|^{1/6}/2.$
\end{theorem} This construction gives another answer to Conjecture \ref{conj:Gutowski}. We are not aware of an explicit construction achieving a better exponent than $1/6$. 

A related question is the following. 

\begin{question}
    Let $f_{\diomega}(n)$ denote the largest possible value of the clique number of an $n$-vertex tournament. Can we determine $f_{\diomega}(n)$ at least asymptotically?
\end{question}

At an earlier stage, we were led to study the families $(B_k)_{k\ge 0}$ and $(G_t)_{t \geq 0}$ as possible extremal families for estimating $f_{\diomega}(n)$.  We later realized, however, that a construction of Alon, Pach, and Solymosi~\cite[Theorem~1.3]{AlonPachSolymosi2001}  gives stronger bounds on $f_{\diomega}(n)$, placing it between $n^{1/3}/\log^2 n$ and $e n^{1/2}$. Nevertheless, the families $(B_k)_{k \geq 0}$ and $(G_t)_{t \geq 0}$ remain natural explicit families in this context, and it is still conceivable that the true values of $\diomega(B_k)$ and $\diomega(G_t)$  are much larger than the lower bound obtained in Theorems~\ref{thm:diomega_Bk} and \ref{thm:gt}. Note also that the construction of Alon, Pach, and Solymosi relies on probabilistic tools, and therefore does not answer the conjecture of Gutowski and Rams.

\section{Miscellaneous results and open questions}\label{sec:open_questions}

We start this section with two more results on the partial order decomposition number of digraphs and tournaments. We first prove that acyclic digraphs can have arbitrarily large $pod$. 
\begin{theorem}
    For every positive integer $k$, there exists an acyclic digraph $D$ such that $pod(D) \geq k$.
\end{theorem}

\begin{proof}
    Let $G$ be a triangle-free graph with chromatic number $2^k$, and let $D$ be an acyclic orientation of $D$. Then given a partial order decomposition $D = D_1 \cup \dots \cup D_m$, colour each arc $a \in A(D)$ with the smallest $i$ such that $a \in D_i$. Note that two arcs forming a directed three-vertex path must receive different colours since $UG(D)$ is triangle-free. Thus, this colouring is in fact a proper colouring of $UG(\arrowL(D))$. But this implies by Theorem \ref{thm:logcoloring} that $m \geq \chi(UG(\arrowL(D))) \geq \log_2(\chi(UG(D))) = \log_2(2^k) = k$. Thus, every partial order decomposition of $D$ requires at least $k$ partial orders, so $pod(D) \geq k$.
\end{proof}

Let $T$ be a tournament. A \emph{dominating set} of $T$ is a set of vertices $S$ such that $V(T) = S \cup N^+(S)$, and the \emph{domination number $dom(T)$} of $T$ is the size of a smallest dominating set of $T$. We prove that a tournament with bounded $pod$ also has bounded domination number. 

\begin{theorem}
     There is a function $f$ such that, for every tournament $T$,  $dom(T) \leq f(pod(T))$.
\end{theorem}

\begin{proof}
    The Erd\H{o}s-Sands-Sauer-Woodrow-Bousquet-Lochet-Thomassé Theorem~\cite{Erdos_Sauer_Woodrow} states that a $k$-arc coloured tournament has a set $S$ of at most $g(k)$ vertices such that, for every  $v \in V(T)$, there is a monochromatic directed path from $S$ to $v$. By definition of $pod$, a tournament has an arc covering by $pod(T)$ posets $P_1,\dots,P_\ell$. We colour each arc $(u,v)$ of $T$ by the nonempty subset $X \subseteq [\ell]$ such that $(u,v) \in A(P_i)$ for every $i \in X$. Each colour class induces a poset as an intersection of posets.
    We deduce a $(2^{pod(T)}-1)$-arc colouring where each colour class is a poset, implying that $S$ is a dominating set.
\end{proof}

We end with some open questions on tournaments and partial order decompositions.

\begin{question}
Is it true that there is a function $f$ such that if tournament $T$ has  $pod(T) \geq f(k,s)$, then it contains a subtournament of size at most $k$ and $pod$ at least $s$? 
\end{question}

\begin{question}
For which tournaments $H$ do $H$-free tournaments have bounded pod?
\end{question}

\begin{question}
    What is the complexity of computing or approximating $pod(T)$ for a tournament $T$?
\end{question}

\bibliographystyle{plain}
\bibliography{bib}

\end{document}